\documentclass[preprint]{jmd}
\usepackage{amssymb}
\UseLinks
\theoremstyle{plain}
\newtheorem{theorem}{Theorem}[section]

\newtheorem{proposition}[theorem]{Pro\-po\-si\-ti\-on}

\theoremstyle{definition}
\newtheorem{definition}[theorem]{Definition}

\newtheorem{remark}[theorem]{Remark}

\renewcommand{\phi}{\varphi}
\renewcommand{\tilde}{\widetilde}
\def\etc.{\emph{et\thinspace c.}}
\DeclareMathOperator{\Id}{Id}

\DeclareMathOperator{\vol}{vol}

\def\C{\mathbb{C}}

\def\R{\mathbb{R}}

\def\H{\mathbb{H}}
\def\Z{\mathbb{Z}}

\def\SS{\mathbb{S}}

\def\Dc{\mathcal{D}}

\newcommand{\Lie}{\mathcal{L}}

\def\dfn{\mathbin{{:}{=}}}
\def\nfd{\mathbin{{=}{:}}}
\newcommand{\at}[1]{{}_{\displaystyle|_{\scriptstyle{#1}}}}
\newcommand{\aeq}{\stackrel{\scriptscriptstyle\text{a.e.\!}}{=}}
\newcommand{\rest}[1]{{}_{\displaystyle\restriction_{\scriptstyle{#1}}}}
\def\norm#1{\|#1\|}
\def\conorm#1{{\norm{#1}^*}}

\begin{document}
\title[Longitudinal foliation rigidity and Lipschitz-continuous invariant forms]{Longitudinal foliation rigidity and Lipschitz-continuous invariant forms for hyperbolic flows}
\author{Patrick Foulon}
\address{Institut de Recherche Mathematique A\-van\-c\'ee\\
UMR 7501 du Centre National de la Recherche Scientifique\\
7 Rue Ren\'e Descartes\\
67084 Strasbourg Cedex\\
France}
\email{foulon@math.u-strasbg.fr}
\author{Boris Hasselblatt}
\address{Department of Mathematics\\
Tufts University\\
Medford, MA 02155\\
USA}
\email{Boris.Hasselblatt@tufts.edu}
\begin{abstract}
  In several contexts the defining invariant structures of a hyperbolic
dynamical system are smooth only in systems of algebraic origin (smooth
rigidity), and we prove new results of this type for a class of flows.

For a compact Riemannian manifold and a uniformly quasiconformal
transversely symplectic $C^2$ Anosov flow we define the \emph{longitudinal
  KAM-cocycle} and use it to prove a rigidity result: The joint
stable/unstable subbundle is Zygmund-regular, and higher regularity implies
vanishing of the longitudinal KAM-cocycle, which in turn implies that the
subbundle is Lipschitz-continuous and indeed that the flow is smoothly
conjugate to an algebraic one. To establish the latter, we prove results
for algebraic Anosov systems that imply smoothness and a special structure
for any Lipschitz-continuous invariant 1-form.

Several features of the reasoning are interesting: The use of exterior
calculus for Lipschitz-continuous forms, that the arguments for geodesic
flows and infranilmanifoldautomorphisms are quite different, and the need
for mixing as opposed to ergodicity in the latter case.
\end{abstract}
\maketitle
\section{Introduction}
\subsection{Statement of main result}
For Anosov systems (\Ref{DEFcanonicaltimechange}), both diffeomorphisms and
flows, interesting phenomena of smooth and geometric rigidity have been
observed in connection with the degree of (transverse) regularity of the
(weak) stable and unstable subbundles of these systems. The seminal result
was the study of volume-preserving Anosov flows on 3-manifolds by Hurder
and Katok \cite{HurderKatok}, which showed that the weak-stable and
weak-unstable foliations are $C^{1+\text{Zygmund}}$ and that there is an
obstruction to higher regularity whose vanishing implies smoothness of
these foliations. This, in turn, happens only if the Anosov flow is
smoothly conjugate to an algebraic one. The cocycle obstruction described
by \emph{K}atok and Hurder was first observed by \emph{A}nosov and is the
first nonlinear coefficient in the \emph{M}oser normal form. Therefore one
might call it the \emph{KAM}-cocycle. This should not be confused with
``KAM'' as in ``Kolmogorov--Arnold--Moser'', and Hurder and Katok refer to
this object as the \emph{Anosov-cocycle}.

In \cite{FH} we showed some analogous rigidity features associated with the
\emph{longitudinal} direction, \ie associated with various degrees of
regularity of the sum of the \emph{strong} stable and unstable subbundles:
For a volume-preserving Anosov flow on a 3-manifold the strong stable and
unstable foliations are Zygmund-regular \cite[Section II.3,
  (3\textperiodcentered1)]{Zygmund}, see \Ref{defregularity}, and there is
an obstruction to higher regularity, which admits a direct geometric
interpretation and whose vanishing implies high smoothness of the joint
strong subbundle and that the flow is either a suspension or a contact
flow. In the papers announced here \cite{FHII,FHIII} we push this to
higher-dimensional systems:
\begin{theorem}\Label{THMMain}
Let $M$ be a compact Riemannian manifold of dimension at least 5, $k\ge2$,
$\varphi\colon\R\times M\to M$ a uniformly quasiconformal
(\Ref{DEFUQC+Symm}) transversely symplectic $C^k$ Anosov flow.

Then $E^u\oplus E^s$ is Zygmund-regular and there is an obstruction to
higher regularity that defines the cohomology class of a cocycle we call
the\/ \emph{longitudinal KAM-cocycle}. This obstruction can be described
geometrically as the curvature of the image of a transversal under a return
map, and
the following are equivalent:
\begin{enumerate}
\item $E^u\oplus E^s$ is ``little Zygmund'' (see \Ref{defregularity}).\label{itemZygmund}
\item The longitudinal KAM-cocycle is a coboundary.\label{itemKAMcobdy}
\item $E^u\oplus E^s$ is Lipschitz-continuous.\label{itemLipschitz}
\item $\varphi$ is up to finite covers, constant rescaling and a
  \emph{canonical time-change} (\Ref{DEFcanonicaltimechange})
  $C^k$-conjugate to the suspension of a symplectic Anosov automorphism of
  a torus or the geodesic flow of a real hyperbolic
  manifold.\label{itemRigidity}
\end{enumerate}
\end{theorem}
To show that \ref{itemLipschitz}.\ implies \ref{itemRigidity}.\ we study
the \emph{canonical 1-form} (\Ref{DEFcanonicaltimechange}) of the
time-change of a geodesic flow or of the suspension of an
infranilmanifoldautomorphism, and because we only have Lipschitz-continuity
at our disposal, we need to explore how smooth-rigidity results can be
pushed to the lowest conceivable regularity. This requires two main
results. On one hand, a Lipschitz-continuous 1-form whose exterior
derivative is invariant under the geodesic flow of a negatively curved
locally symmetric space must be (a constant multiple of) the canonical
1-form for the flow. On the other hand, an essentially bounded 2-form
invariant under an infranilmanifoldautomorphism is smooth. A special case
of the second result is that in which the 2-form arises as the exterior
derivative of a Lipschitz-continuous 1-form, in which case it vanishes.
Thus, both results involve exterior calculus of Lipschitz-continuous
1-forms. It is clear that this has to be done with care, and we invoke
results to the effect that, for instance, the classical Stokes Theorem
holds for Lipschitz-continuous forms \cite{Dubrovskiy}.

Given this common motivation, it is surprising that our two separate
results involve rather different arguments for geodesic flows on one hand
and suspensions on the other hand.

A particular point of interest is in this respect that while, like with
many other results in hyperbolic dynamics, ergodic theory enters the proof
in the case of geodesic flows only to the extent that we use ergodicity of
the geodesic flow, it turns out that for the case of a suspension we use in
an essential way that the infranilmanifoldautomorphism is indeed mixing
rather than merely ergodic. This reflects the need to deal with parabolic
effects due to the nilpotent part.

\subsection{Background and terminology}
We now introduce the notions that play a role in this result and the proof.
\begin{definition}[\cite{KatokHasselblatt}]\Label{DEFcanonicaltimechange}
An Anosov flow on a manifold $M$ is a smooth flow $\varphi^t$ with 
\begin{itemize}
\item an
invariant decomposition $TM=X\oplus E^u\oplus E^s$ (where
$X=\dot\varphi\neq0$ is the generator of the flow and $E^u$ and $E^s$ are
called the unstable and stable subbundles) and 
\item a Riemannian metric on $M$
such that $D\varphi^t\rest{E^s}$ and $D\varphi^{-t}\rest{E^u}$ are
contractions whenever $t>0$.
\end{itemize}
The definition of Anosov diffeomorphism is
analogous with $t\in\Z$ and $X$ absent. 

The \emph{canonical 1-form} $A$ of an Anosov flow $\varphi^t$ is defined by
$A(X)=1$ and $E^u,E^s\subset\ker A$.
A \emph{canonical time-change} is defined using a closed
1-form $\alpha$ by replacing the generator $X$ of the flow by the vector
field $X/(1+\alpha(X))$, provided $\alpha$ is such that the denominator is
positive. (See \Ref{Scantimechg} for more on canonical time-changes.)
\end{definition}
The subbundles are invariant and (H\"older-) continuous with
smooth integral manifolds $W^u$ and $W^s$ that are coherent in that $q\in
W^u(p)\Rightarrow W^u(q)=W^u(p)$. $W^u$ and $W^s$ define laminations
(continuous foliations with smooth leaves).
\begin{definition}\Label{defregularity}
A function $f$ between metric spaces is said to be \emph{H\"older
continuous} if there is an $H>0$, called the H\"older exponent, such that
$d(f(x),f(y))\le\text{const.} d(x,y)^H$ whenever $d(x,y)$ is sufficiently
small. We specify the exponent by saying that a function is $H$-H\"older. A
continuous function $f\colon U\to L$ on an open set $U\subset L'$ in a
normed linear space to a normed linear space is said to be \emph{Zygmund-regular}
if there is $Z>0$ such that $\|f(x+h)+f(x-h)-2f(x)\|\le Z\|h\|$ for all
$x\in U$ and sufficiently small $\|h\|$. To specify a value of $Z$ we may
refer to a function as being $Z$-Zygmund. The function is said to be
\emph{``little Zygmund''} (or ``zygmund'') if $\|f(x+h)+f(x-h)-2f(x)\|=o(\|h\|)$.
For maps between manifolds these definitions are applied in smooth local
coordinates.
\end{definition}
Zygmund regularity implies modulus of continuity $O(|x\log|x||)$ and hence
$H$-H\"older continuity for all $H<1$ \cite[Section II.3,
  Theorem~(3\textperiodcentered4)]{Zygmund}. It follows from
Lipschitz-continuity and hence from differentiability. Being ``little
Zygmund'' implies having modulus of continuity $o(|x\log|x||)$ and follows
from differentiability but not from Lipschitz-continuity.

The regularity of the unstable subbundle $E^u$ is usually substantially
lower than that of the weak-unstable subbundle $E^u\oplus E^\varphi$. The
exception are geodesic flows, where the strong unstable subbundle is
obtained from the weak-unstable subbundle by intersecting with the kernel
of the invariant contact form. This has the effect that the strong-unstable
and weak-unstable subbundles have the same regularity. However, time
changes affect the regularity of the strong-unstable subbundle, and this is
what typically keeps its regularity below $C^1$. In \cite{FH} we
presented a \emph{longitudinal KAM-cocycle} that is the obstruction to
differentiability, and we derived higher regularity from its vanishing.
\begin{theorem}[{\cite[Theorem 3]{FH}}]\Label{THMOldMain}
Let $M$ be a 3-manifold, $k\ge2$, $\varphi\colon\R\times M\to M$ a $C^k$
volume-preserving Anosov flow.
Then $E^u\oplus E^s$ is Zygmund-regular, and
there is an obstruction to higher regularity that can be described
geometrically as the curvature of the image of a transversal under a return
map.
This obstruction defines the cohomology class of a cocycle (the
longitudinal KAM-cocycle), and the following are equivalent:
\begin{enumerate}
\item $E^u\oplus E^s$ is ``little Zygmund'' (see \Ref{defregularity}).
\item The longitudinal KAM-cocycle is a coboundary.
\item $E^u\oplus E^s$ is Lipschitz-continuous.
\item $E^u\oplus E^s\in C^{k-1}$.
\item $\varphi$ is a suspension or contact flow.\label{itemIsSuspension}
\end{enumerate}
\end{theorem}
In \ref{itemIsSuspension}.\ no stronger rigidity should be expected because
$E^u\oplus E^s$ is smooth for all suspensions and contact flows. See
\cite{Paternain,DairbekovPaternain} for applications of this to magnetic
flows.

The work by Hurder and Katok in \cite{HurderKatok} inspired developments of
substantial extensions to higher dimensions, see, for example, \cite{Hyp}.
The present work extends our earlier work to higher-dimensional systems in
this ``longitudinal'' context. This requires somewhat stringent
assumptions, however.
\begin{definition}\Label{DEFUQC+Symm}
An Anosov flow is said to be \emph{uniformly quasiconformal} if
\begin{equation}\label{eqquasiconformality}
K_i(x,t)\dfn\frac{\norm{d\varphi^t\rest{E^i}}}{\conorm{d\varphi^t\rest{E^i}}}
\end{equation}
is bounded on $\{u,s\}\times M\times\R$, where
$\conorm{A}\dfn\min_{\|v\|=1}\|Av\|$ is the \emph{conorm} of a linear map $A$.
\end{definition}
\subsection{Rigidity}
The proof that
\ref{itemZygmund}.$\Rightarrow$\ref{itemKAMcobdy}.$\Rightarrow$\ref{itemLipschitz}.\ in
\Ref{THMMain} largely follows the line of reasoning already presented in
\cite{FH} and appears in \cite{FHII}.

In the 3-dimensional case we showed that smoothness of $E^u\oplus E^s$
implies that $\varphi$ is a suspension or contact flow, but in the present
situation we obtain more detailed information because of the
quasiconformality-assumption. This uses a rigidity theorem by Fang:
\begin{theorem}[{\cite[Corollary 3]{FangCo}}]\Label{THMFang}
Let $M$ be a compact Riemannian manifold and $\varphi\colon\R\times M\to M$ a
transversely symplectic Anosov flow with $\dim E^u\ge2$ and $\dim E^s\ge2$.
Then $\varphi$ is quasiconformal if and only if $\varphi$ is up to finite
covers $C^\infty$ orbit equivalent either to the suspension of a symplectic
hyperbolic automorphism of a torus, or to the geodesic flow of a closed
hyperbolic manifold.
\end{theorem}
This also serves to illustrate that the assumption of uniform
quasiconformality is quite restrictive. We should also point out that our
result about rigidity of the situation with $E^u\oplus E^s\in C^1$ overlaps
with a closely related one by Fang, although the proof is independent:
\begin{theorem}[{\cite[Corollary 2]{FangCo}}]\Label{THMFang2}
Let $\varphi$ be a $C^\infty$ volume-preserving quasiconformal Anosov flow.
If $E^s\oplus E^u\in C^1$ and $\dim E^u\ge3$ and $\dim E^s\ge2$ (or $\dim
E^s\ge3$ and $\dim E^u\ge2$), then $\varphi$ is up to finite covers and a
constant change of time scale $C^\infty$ flow equivalent either to the
suspension of a hyperbolic automorphism of a torus, or to a canonical time
change (\Ref{DEFcanonicaltimechange}) of the geodesic flow of a closed
hyperbolic manifold.
\end{theorem}
\Ref{THMFang} yields ``\ref{itemLipschitz}.$\Rightarrow$\ref{itemRigidity}.''\ in
\Ref{THMMain} due to the following results.
\begin{theorem}\Label{THMRigidityGeod}
Let $M$ be a compact locally symmetric space with negative sectional
curvature and consider a time-change of the geodesic flow whose canonical
1-form is Lipschitz-continuous. Then the canonical form of the time-change
is $C^\infty$, and the time-change is a canonical time-change.
\end{theorem}
\begin{theorem}\Label{THMRigiditySusp}
Let $\psi$ be a hyperbolic automorphism of a torus or a nilmanifold
$\Gamma\backslash M$ and consider a time-change of the suspension whose
canonical 1-form is Lipschitz-continuous. Then the canonical form of the
time-change is $C^\infty$, and the time-change is a canonical time-change.
\end{theorem}
Indeed, by \Ref{THMFang}, $\varphi$ is smoothly orbit equivalent to either
the geodesic flow of a real hyperbolic manifold or a suspension of a
symplectic automorphism of an $n$-torus. To show that the flow is, after
rescaling, smoothly \emph{conjugate} to one of these models, use the orbit
equivalence to regard the canonical form for $\varphi$ as an invariant form
for the algebraic system and then apply \Ref{THMRigidityGeod} or
\Ref{THMRigiditySusp}.

After introducing some background on canonical time-changes, we outline how
to establish \Ref{THMRigidityGeod} and \Ref{THMRigiditySusp}.
\section{Canonical time-changes}\Label{Scantimechg}
We make a few remarks here about canonical time-changes
(\Ref{DEFcanonicaltimechange}) because these are not frequently encountered
in the literature; they may serve to show why this is a natural class of
time-changes to expect in rigidity results for flows.
\begin{proposition}[Trivial time-changes]\Label{PRPtrivtimechg}
Consider a flow $\varphi^t$ generated by the vector field $X$ and a smooth
function $f\colon M\to\R$ such that $1+df(X)>0$. Then $\Psi\colon
x\mapsto\varphi^{f(x)}(x)$ conjugates the flow generated by the vector
field $X_f\dfn\dfrac X{1+df(X)}$ to $\varphi^t$.
\end{proposition}
\begin{proof}
Smoothness of $f$ and $1+df(X)>0$ ensure that $\Psi$ is a diffeomorphism.
Now we write $x_t=\varphi^t(x)$ and use the chain rule to compute
\begin{multline*}
d\Psi(X(x))
=
\frac{d\Psi}{dt}
=
\frac{d}{dt}
\varphi^{f(x_t)}(x_t)\at{t=0}
=
\frac{d\varphi}{dt}\at{t=0}df(X)(x)+X(\varphi^{f(x)}(x))
\\=
X(\varphi^{f(x)}(x))\cdot df(X)(x)+X(\varphi^{f(x)}(x))
=
(1+df(X)(x))X(\varphi^{f(x)}(x)),
\end{multline*}
which gives 
$d\Psi(X_f)=X$ upon division by $1+df(X)(x)$.
\end{proof}
\begin{proposition}[Cohomology class]
If $\alpha$ and $\beta$ are cohomologous closed 1-forms with
$1+\alpha(X)>0$ and $1+\beta(X)>0$ then the associated canonical
time-changes of $X$ are smoothly conjugate.
\end{proposition}
\begin{remark}
This tells us that the cohomology class of $\alpha$ is the material
ingredient in a canonical time-change by $\alpha$.
\end{remark}
\begin{proof}
Writing $\beta=\alpha+df$ with smooth $f$ we observe that
$$
\frac X{1+\beta(X)}
=
\frac X{1+\alpha(X)+df(X)}
=
\frac{\frac X{1+\alpha(X)}}{1+df(\frac X{1+\alpha(X)})}
=
\Big(\frac X{1+\alpha(X)}\Big)_f.
$$
Now use \Ref{PRPtrivtimechg}.
\end{proof}
\begin{proposition}[Regularity]\Label{PRPcanformforcantimechg}
Suppose $X_0$ generates an Anosov flow, and $\alpha$ is a closed 1-form
such that $1+\alpha(X_0)>0$. If $A_0$ denotes the canonical form for $X_0$
then $A\dfn A_0+\alpha$ is the canonical form for
$X\dfn\frac{X_0}{1+\alpha(X_0)}$.
\end{proposition}
\begin{remark}
This shows, in particular, that canonical time-changes with smooth closed
forms do not affect the regularity of the canonical form.
\end{remark}
\begin{proof}
We first note that two invariant 1-forms for an Anosov flow are
proportional: Both being constant on $X$, this follows from the fact that a
continuous 1-form that vanishes on $X$ is trivial \cite[Lemma 1]{FK}.

Since $\alpha$ is closed we have $dA=dA_0$. Also
$$
A(X)=\frac{A_0(X_0)+\alpha(X_0)}{1+\alpha(X_0)}=1,
$$ which implies that $\Lie_XA=0$, \ie $A$ is $X$-invariant and hence
proportional to the canonical 1-form of $X$. But
$A(X)=1$ then implies that $A$ is equal to the canonical 1-form of $X$.
\end{proof}
\section{Rigidity results}
The results that imply \Ref{THMRigidityGeod} and \Ref{THMRigiditySusp} are
of independent interest and will therefore appear in a separate publication
from the application \cite{FHIII} to quasiconformal Anosov flows.
\begin{theorem}\Label{THMGeodesicmain}
Let $M$ be a compact locally symmetric space with negative sectional
curvature and suppose $A$ is a Lipschitz-continuous 1-form such that
$dA$ is invariant under the geodesic flow. Then $A$ is
$C^\infty$, and indeed $dA$ is a constant multiple of the exterior
derivative of the canonical 1-form for the geodesic flow.
\end{theorem}
\begin{remark}\Label{REMessbdd}
Note that the Lipschitz assumption ensures that $dA$ is defined almost
everywhere and essentially bounded \cite{GKS}. This is all we use. For
comparison, we state an earlier result of Hamenst\"adt: 
\end{remark}
\begin{theorem}[{\cite[Theorem A.3]{Hamenstadt}}]
If the Anosov splitting of the geodesic flow of a compact negatively curved
manifold is $C^1$ and $A$ is a $C^1$ 1-form such that $dA$ is invariant,
then $dA$ is proportional to the exterior derivative of the canonical
1-form of the geodesic flow.
\end{theorem}
\begin{proof}[Proof of\/ \Ref{THMRigidityGeod} from \Ref{THMGeodesicmain}]
The hypotheses of\/ \Ref{THMRigidityGeod} and of \Ref{THMGeodesicmain} imply
smoothness; we need to show that the vector field $X$ that generates the
time-change agrees with a canonical time-change of a constantly scaled
version of $X_0$, where $X_0$ generates the geodesic flow. Rescale $X_0$ to
$X_0/\kappa$, where $\kappa\in\R$ is defined by $dA=\kappa dB$ and then apply
the canonical time-change defined by the 1-form $\overbar\alpha\dfn A-\kappa
B$. Since the resulting vector field $\dfrac{X_0}{\kappa+\overbar\alpha(X_0)}$
is a scalar multiple of $X$, the claim follows from
$$
A\Big(\frac{X_0}{\kappa+\overbar\alpha(X_0)}\Big)
=
\frac{A(X_0)}{\kappa B(X_0)+(A-\kappa B)(X_0)}
=1,
$$
where we used $B(X_0)=1$. That the last denominator is $A(X_0)$ and hence
positive justifies the use of $\overbar\alpha$ to define a canonical time-change.
\end{proof}
\begin{theorem}\Label{THMSuspensionmain}
Let $\psi$ be a hyperbolic automorphism of a torus or a infranilmanifold
$\Gamma\backslash M$. Then any essentially bounded invariant 2-form is
almost everywhere equal to an $M$-invariant (hence smooth) closed 2-form.

If, in addition, the form is exact, then it vanishes almost everywhere.
\end{theorem}
\begin{remark}
We point out that in this proof we use that the automorphism is mixing
(rather than just ergodic). The need for this is an interesting side-light
on how parabolic effects enter into our considerations.
\end{remark}
\begin{proof}[Proof of \Ref{THMRigiditySusp} from \Ref{THMSuspensionmain}]
Denote by $A$ the canonical form of the time-change and by $B$ the
canonical form of the suspension. \Ref{THMSuspensionmain} applied to $A$
implies that $A$ is smooth and closed, and hence so is $\alpha\dfn A-B$
since $B$ is also closed. Writing $X_0$ for the suspension vector field we
find that the canonical time-change
$$
X\dfn\frac{X_0}{1+\alpha(X_0)}=\frac{X_0}{A(X_0)}
$$
of $X_0$ is the given vector field since by construction $A(X)\equiv1$.
\end{proof}
\section{Proof of \Ref{THMGeodesicmain}}
Using exterior calculus (carefully!) we show that $A$ is a contact form,
\ie $A\wedge\bigwedge_{i=1}^ndA$ is a volume, and that $X_0$ is in the
kernel of $dA$. By duality, for every $\xi$ there is a $\psi(\xi)$ such that 
$$dA(\xi,\cdot)=dB(\psi(\xi),\cdot);$$this is defined whenever $dA$ is, and
we choose $\psi(X_0)=X_0$ and $\psi(\ker B)\subset\ker B$. We next show
that $\psi\aeq \kappa\Id+N$, where $\kappa\in\R\smallsetminus\{0\}$ and $N$
is a nilpotent operator. The main effort is now directed at showing that
$N=0$, \ie that $dA\aeq \kappa\,dB$ (smoothness of $A$ can then be obtained
via some delicate exterior calculus). To that end we $L^1$-approximate $N$
by a continuous operator, take the Birkhoff average of this approximation
(which does not change the $L^1$-distance to $N$) and show that it is
defined and continuous everywhere and so intertwined with the flow that one
can apply arguments from \cite{BFL1,BFL} to conclude that it vanishes
identically. Thus $N$ is arbitrarily $L^1$-close to 0 and hence vanishes
itself.

Taking Birkhoff averages of the continuous $L^1$-approximation $F$ of $N$
requires substantial technical underpinnings because the Birkhoff Ergodic
Theorem applies to scalar functions. Therefore we show that we can choose a
measurable orthonormal frame field for $E^u$ on $M$ that consists of vector
fields $\xi$ chosen in such a way that
\begin{itemize}
\item$\xi$ is continuous and $\Dc$-parallel along unstable leaves that are
  homeomorphic to Euclidean space, \ie nonperiodic leaves,
\item$\xi\in E^j$ for $j=1$ or $j=2$,
\item if $\xi\in E^j$ then the Lie bracket with the generator
$X$ of the geodesic flow is $[X,\xi]=\Dc_X\xi+j\xi=j\xi$ (the last equality
uses that $\xi$ is $\Dc$-parallel) and hence $\gamma^t(\xi)=e^{jt}\xi$.
\end{itemize}
Here $\Dc$ is a Kanai connection constructed for this purpose (as in
\cite{BFL1,BFL}), and its properties produce such a frame field. We use
that the geodesic flow of a locally symmetric space admits an invariant
splitting of the unstable subbundle into fast- and slow-unstable bundles
corresponding to the exponents $1$ and $2$; we write $E^u=E^1\oplus E^2$.
In the constant-curvature case we have $E^2=0$.

We next show that the Birkhoff average of $F$ has a continuous extension
$\tilde F$ to the entire manifold that is parallel along stable and
unstable manifolds. We then follow arguments in \cite{BFL1,BFL} to show
that this implies $\tilde F=0$ and hence $N=0$.
\section{Proof of\/ \Ref{THMSuspensionmain}}
Consider a hyperbolic automorphism $\psi$ of either a torus or a
nilmanifold $\Gamma\backslash N$. Suppose $\omega$ is an essentially
bounded 2-form such that $\psi_*\omega=\omega$.

One can write a 2-form \emph{locally} as $\sum_{1\le i<j\le
  n}a_{ij}dx^i\wedge dx^j$, and we will look for a way of doing so
\emph{globally} and with constant coefficients. To that end we pass to the
complexification  of the tangent bundle and work with a basis of
$N$-invariant sections $X_i$.

At one point we choose a basis in such a way that $\psi$ is in Jordan
canonical form with respect to the dual basis consisting of the forms
$X_i^*$, \ie $\psi=\sum_i a_{ij}X_i^*$ with $A=(a_{ij})$ in Jordan form
(strictly triangular since we passed to the complexification). Now
translate the basis and the dual basis by $N$ to get invariant sections.
With these choices, a Jordan block for eigenvalue $\lambda_l$ for the $n$th
iterate is of the form
$
A_l^n=\lambda_l^nM_{\ell,\lambda_l}(n)
$
with $
M_{\ell,\lambda_l}(n)$ a fixed polynomial in $n$.

Now denote by $\Omega(x)$ the matrix that represents $\omega_x$ with
respect to the invariant frame field. Then the matrix of $\psi_*\omega$ is
given by ${}^t\!\!A\,\Omega(x)A$ and hence the iterated relation
$\omega=\psi^n_*\omega$ becomes $\Omega(\psi^n(x))={}^t\!\!A^n\Omega(x)A^n$,
which is bounded (in $n$) for almost every $x$ (since $\omega$ is
essentially bounded). Fix such an $x$ and decompose $\Omega$ into
(not necessarily square or diagonal) blocks $\Omega_{ij}$ according to
the Jordan form of $A$, \ie in such a way that 
\[
\Omega_{ij}(\psi^n(x))={}^t\!\!A_i^n\Omega_{ij}(x)A_j^n=
(\lambda_i\lambda_j)^n\underbrace{{}^t\!M_{\ell_i,\lambda_i}(n)\;\Omega_{ij}(x)\,M_{\ell_j,\lambda_j}(n)}_{\textstyle\nfd P_{ijx}(n)},
\]
where $\ell_i$ and $\ell_j$ are the sizes of the blocks $A_i$ and $A_j$,
respectively. For any $i$, $j$ and $x$, $P_{ijx}(n)$ is a matrix-valued
polynomial in $n$, and indeed it is constant:
\begin{itemize}
\item If $|\lambda_i\lambda_j|\neq1$ then  $P_{ijx}(n)=0$---otherwise $|(\lambda_i\lambda_j)^nP_{ijx}(n)|$ grows exponentially and is, in
  particular, unbounded.
\item If $|\lambda_i\lambda_j|=1$ then $P_{ijx}(n)$ is constant (in
  $n$)---otherwise $|(\lambda_i\lambda_j)^nP_{ijx}(n)|=|P_{ijx}(n)|$ is
  unbounded.
\end{itemize}
We therefore have
\[
{}^t\!M_{\ell_i,\lambda_i}(n)\Omega_{ij}(x)M_{\ell_j,\lambda_j}(n)=P_{ijx}(n)=P_{ijx}(0)=\Omega_{ij}(x),
\]
so $\Omega_{ij}(\psi^n(x))=(\lambda_i\lambda_j)^n\Omega_{ij}(x)$ for
all $i$, $j$. 

This shows that every entry of the matrix $\Omega$ is almost everywhere
equal to an eigenfunction of $\psi$. Since $\psi$ is mixing, all
eigenfunctions are constant, and hence $\Omega$ is almost everywhere equal
to an $M$-invariant (hence smooth) 2-form.

To prove the last assertion of \Ref{THMRigiditySusp}, \ie that $\Omega$
vanishes if it is exact, we introduce a notion of averaging. Let
$$\vol_\C\dfn X_1^*\wedge\dots\wedge X_n^*$$ be the $N$-invariant
complex-valued volume form defined by the dual basis we used before. By
compactness, this gives a finite volume. 
%
Then
for any $p$-form
$$
\alpha=\sum_{i_1,\dots,i_p}\alpha_{i_1,\dots,i_p}X_{i_1}^*\wedge\dots\wedge X_{i_p}^*
$$
we define the \emph{average}
$$
\overbar\alpha\dfn\sum_{i_1,\dots,i_p}
(\int_{\Gamma\backslash N}\alpha_{i_1,\dots,i_p}\,d\vol_\C)
X_{i_1}^*\wedge\dots\wedge X_{i_p}^*
$$
and prove that 
 $\overbar{d\alpha}=d\overbar\alpha$ for any 1-form $\alpha$.
We can write
$$
\psi_*(\sum_i\overbar\alpha_iX_i^*)
=
\sum_i\overbar\alpha_i\psi_*(X_i^*)
=
\sum_i(\overbar\alpha_ia_{ij})X_j^*,
$$
and since the coefficients
here are constant, we obtain
$d\psi_*\overbar\alpha=\psi_*d\overbar\alpha$.

If $\omega$ is an exact $\psi$-invariant 2-form with constant coefficients,
then we write $\omega=d\alpha$ and note that
$$
\psi_*d\overbar\alpha
=
\psi_*\overbar{d\alpha}
=
\psi_*\overbar\omega
=
\psi_*\omega
=
\omega
=
\overbar\omega
=
\overbar{d\alpha}
=
d\overbar\alpha,
$$
\ie$d(\psi_*\overbar\alpha-\overbar\alpha)=\psi_*d\overbar\alpha-d\overbar\alpha=0$. Thus, there is an $f$ such
that $\psi_*\overbar\alpha-\overbar\alpha=df$ and, in particular,
$$
\sum\overbar\alpha_ia_{ij}-\overbar\alpha_j=df(X_j).
$$
Since, on the other hand, $\int_{\Gamma\backslash N}df(X_j)\,d\vol_\C=0$
and the integrand is constant, we have
$\sum\overbar\alpha_ia_{ij}=\overbar\alpha_j$, \ie $\overbar\alpha$ is a
$\psi$-invariant 1-form. But then, hyperbolicity of
$\psi$ implies that $\overbar\alpha=0$ (see, \eg \cite[Lemma 1]{FK}) and hence $\omega=\overbar\omega=\overbar{d\alpha}=d\overbar\alpha=0$.

\end{document}